\documentclass[12pt]{amsart}
\usepackage{url}
\usepackage{amssymb, amscd}
\usepackage{xy}
\usepackage{bm}
\xyoption{all}

\usepackage{pinlabel}
\usepackage[linktocpage]{hyperref}

\newcommand{\Z}{\mathbb Z}

\def\spinc{\ifmmode{\textrm{Spin}^c}\else{$\textrm{Spin}^c$}\fi}
\newcommand{\spincs}{\mathfrak s}

\def\oz{Ozsv{\'a}th--Szab{\'o}}
\renewcommand{\phi}{\varphi}

\newtheorem{theorem}{Theorem}[section]

\newtheorem{lemma}[theorem]{Lemma}

\newtheorem{corollary}[theorem]{Corollary}
\newtheorem*{lmprob}{Problem 4.6}
\newtheorem*{ack}{Acknowledgements}

\theoremstyle{definition}

\def\conn{\mathbin{\#}}

\DeclareMathOperator*{\id}{id}

\newcommand{\thmref}[1]{Theorem~\ref{T:#1}}
\newcommand{\lemref}[1]{Lemma~\ref{L:#1}}

\newcommand{\clref}[1]{Claim~\ref{C:#1}}

\title{On smoothly superslice knots}

\author[Daniel Ruberman]{Daniel Ruberman}
\address{Department of Mathematics, MS 050\newline\indent Brandeis
University \newline\indent Waltham, MA 02454}
\email{ruberman@brandeis.edu}
\thanks{The author was partially supported by NSF Grant 1506328.\\
Math.~Subj.~Class.~2010: 57M25 (primary), 57Q60 (secondary).}

\begin{document}
\begin{abstract} 
We find smoothly slice (in fact doubly slice) knots in the 3-sphere with trivial Alexander polynomial that are not superslice, answering a question posed by Livingston and Meier.
\end{abstract}
\maketitle
\section{Introduction}
A recent paper of Livingston and Meier raises an interesting question about {\em superslice knots}.  Recall~\cite{brakes:superslice} that a knot $K$ in $S^3$ is said to be superslice if there is a slice disk $D$ for $K$ such that the double of $D$ along $K$ is the unknotted 2-sphere $S$ in $S^4$. We will refer to such a disk as a {\em superslicing disk}. In particular, a superslice knot is slice and also doubly slice, that is, a slice of an unknotted $2$-sphere in $S^4$.  Livingston and Meier ask about the converse in the smooth category.
\begin{lmprob}[Livingston-Meier \cite{livingston-meier:ds}]\label{prob:lm}
Find a smoothly slice knot $K$ with $\Delta_K(t) = 1$ that is not smoothly superslice.
\end{lmprob}
The corresponding question in the topological (locally flat) category is completely understood~\cite{livingston-meier:ds,meier:ds}, for a knot $K$ with $\Delta_K(t) = 1$  is topologically superslice.

In this note we give a simple solution to problem 4.6, making use of Taubes' proof~\cite{taubes:periodic} that Donaldson's diagonalization theorem~\cite{donaldson} holds for certain non-compact manifolds. For $K$ a knot in $S^3$, we write $\Sigma_k(K)$ for a $k$-fold cyclic branched cover of $S^3$ branched along $K$.  The same notation will be used for the corresponding branched cover along an embedded disk in $B^4$ or sphere in $S^4$.  
\begin{theorem}\label{T:ss}
Suppose that $J$ is a knot with Alexander polynomial $1$, so that $\Sigma_k(J) = \partial W$, where $W$ is simply connected and the intersection form on $W$ is definite and not diagonalizable.  Then the knot $K = J \# -J$ is smoothly doubly slice, but is not smoothly superslice.
\end{theorem}
An unpublished argument of Akbulut says that the positive Whitehead double of the trefoil is a knot $J$ satisfying the hypotheses of the theorem, with $k=2$. The construction is given as~\cite[Exercise 11.4]{akbulut:book} and is also documented, along with some generalizations, in the paper~\cite{cochran-gompf:donaldson}. Hence $J$ gives an answer to Problem 4.6. We remark that for the purposes of the argument, it doesn't matter if $W$ is positive or negative definite, as one could replace $J$ by $-J$ and change all the signs.

We need a simple and presumably well-known algebraic lemma.
\begin{lemma}\label{L:pushout}
Suppose that 
$$
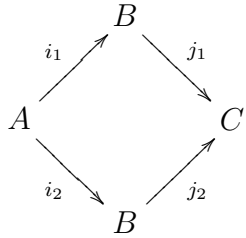
\xymatrix{
& B \ar[dr]^{j_1} &\\
A \ar[ur]^{i_1} \ar[dr]_{i_2}&&  C\\
& B \ar[ur]_{j_2} &
}
$$
is a pushout of groups, and that $i_1 = i_2$. Then $C$ surjects onto $B$. \end{lemma}
\begin{proof}
This follows from the universal property of pushouts; the identity map $\id_B$ satisfies $\id_B \circ i_1 = \id_B \circ i_2$, and hence defines a homomorphism $C \to B$ with the same image as $\id_B$.
\end{proof}
Applying \lemref{pushout} to the decomposition of the complement of the unknot in $S^4$ into two disk complements, we obtain the following useful facts. (The first of these was presumably known to Kirby and Melvin; compare~\cite[Addendum, p. 58]{kirby-melvin:R}, and the second is due to Gordon and Sumners~\cite{gordon-sumners:ball-pairs}.)
\begin{corollary}\label{C:disk}
If $K$ is superslice and $D$ is a superslicing disk, then 
$$
\pi_1(B^4 -D) \cong \Z\ \text{and}\ \Delta_K(t)=1.
$$
\end{corollary}
\begin{proof}
The lemma says that there is a surjection $\Z \cong \pi_1(S^4 - S) \to \pi_1(B^4 -D)$.  Hence $\pi_1(B^4 -D)$ is abelian and so must be isomorphic to $\Z$.  This condition implies, using Milnor duality~\cite{milnor:covering} in the infinite cyclic covering, that the homology of the infinite cyclic covering of $S^3-K$ vanishes, which is equivalent to saying that $\Delta_K(t)=1$.
\end{proof}

\begin{proof}[Proof of \thmref{ss}]
It is standard~\cite{sumners:inv2} that any knot of the form $J \conn -J$ is doubly slice. In fact, it is a slice of the $1$-twist spin of $J$, which was shown by Zeeman~\cite{zeeman:twist} to be unknotted.

Suppose that $K$ is superslice and let $D$ be a superslicing disk, so $D \cup_K D = S$, an unknotted sphere. Then $S^4 = \Sigma_k(S) = V \cup_Y V$, where we have written $Y= \Sigma_k(K)$ and  $V = \Sigma_k(D)$. By \clref{disk}, the $k$-fold cover of $B^4 - D$ has $\pi_1 \cong \Z$, so the branched cover $V$ is simply connected.

Note that $\Sigma_k(K) = \Sigma_k(J) \conn -\Sigma_k(J)$.  Since $\Delta_J(t) =1$, the same is true for $\Delta_K(t)$, moreover this implies that both $\Sigma_k(J)$ and $\Sigma_k(K)$ are homology spheres.  An easy Mayer-Vietoris argument says that $V = \Sigma_k(D)$ is a homology ball; in fact \clref{disk} implies that it is contractible.  Adding a $3$-handle to $V$, we obtain a simply-connected homology cobordism $V'$ from $\Sigma_k(J)$ to itself.  By hypothesis, there is a manifold $W$ with boundary $\Sigma_k(J)$ and non-diagonalizable intersection form. Stack up infinitely many copies of $V'$, and glue them to W to make a definite periodic-end manifold $M$, in the sense of Taubes~\cite{taubes:periodic}.  Since $\pi_1(V)$ is trivial, $M$ is {\em admissible} (see ~\cite[Definition 1.3]{taubes:periodic}), and Taubes shows that its intersection form (which is the same as that of $W$) is diagonalizable. This contradiction proves the theorem.
\end{proof}


The fact that $\pi_1(B^4-D) \cong \Z$ for a superslicing disk leads to a second obstruction to supersliceness, based on the \oz\ $d$-invariant~\cite{oz:boundary}. Recall from~\cite{manolescu-owens:delta} (for degree $2$ covers) and~\cite{jabuka:delta} in general that for a knot $K$ and prime $p$, that one denotes by $\delta_{p^n}(K)$ the $d$-invariant of a particular spin structure $\spincs$ on $\Sigma_{p^n}$ pulled back from the $3$-sphere.  The fact that a $p^n$ fold branched cover of a slicing disk is a rational homology ball implies that if $K$ slice then $\delta_{p^n}(K)= 0$. For a non-prime-power degree $k$, the invariant $\delta_k(K)$ might not be defined, because $\Sigma_k(K)$ is not a rational homology sphere.  (One might define such an invariant using Floer homology with twisted coefficients as in~\cite{behrens-golla:twist,levine-ruberman:correction}, but there's no good reason that it would be a concordance invariant.)
\begin{theorem}\label{T:dss}
If $K$ is superslice, then for any $k$, the $d$-invariant $d(\Sigma_k(K),\spincs_0)$ is defined and vanishes.
\end{theorem}
\begin{proof}
Since by \clref{disk} the Alexander polynomial is trivial, so $\Sigma_k(K)$ is a homology sphere, and hence $d(\Sigma_k(K),\spincs_0)$ is defined. (There is only the one spin structure.)  As in the proof of \thmref{ss}, the branched cover $\Sigma_k(D)$ is contractible, and 
hence~\cite[Theorem 1.12]{oz:boundary},  $d(\Sigma_k(K),\spincs_0) = 0$.
\end{proof}
Sadly, we do not know any examples of a slice knot where \thmref{dss} provides an obstruction to it being superslice. For such a knot would not be ribbon, so we would also have a counterexample to the slice-ribbon conjecture!
\begin{ack}
Thanks to Hee Jung Kim for an interesting conversation that led to this paper, and to Chuck Livingston, Paul Melvin, and Nikolai Saveliev for comments on an initial draft.
\end{ack}
\vspace*{-2ex}
\def\cprime{$'$}
\providecommand{\bysame}{\leavevmode\hbox to3em{\hrulefill}\thinspace}

\end{document}